\documentclass[11pt]{amsart}

\usepackage{indentfirst}
\usepackage[mathscr]{eucal}
\usepackage{amssymb,amsmath,amsfonts}

\usepackage{enumerate}
\usepackage{amsmath}
\usepackage{amssymb}
\usepackage{amstext}
\usepackage[latin1]{inputenc}
\usepackage{latexsym}
\usepackage{amsthm}
\usepackage{amscd}

\newtheorem{theorem}{Theorem}
\newtheorem*{theoremintro}{Theorem}
\newtheorem{lemma}{Lemma}

\newtheorem{corollary}[theorem]{Corollary}

\newtheorem{definition}{Definition}

\numberwithin{equation}{section}

\def\esf{\mathbb{S}}
\def\R{\mathbb{R}}
\def\C{\mathbb{C}}
\def\B{\mathbb{B}}
\def\N{\mathbb{N}}
\def\D{\mathbb{D}}
\def\ee{\textsc{e}}
\def\J{\mathcal{J}}
\def\I{\mathcal{I}}
\def\pp{\mathtt{P}}
\def\cte{\textrm{const\,}}

\newcommand{\til}[1]{\widetilde{#1}}
\newcommand{\gorro}[1]{\widehat{#1}}

\newcommand{\longui}{\operatorname{length}}
\newcommand{\dist}{\operatorname{dist}}
\newcommand{\intc}{\operatorname{Int}}
\newcommand{\diam}{\operatorname{diam}}

\def\a{{\alpha}}

\def\g{{\gamma}}
\def\G{{\Gamma}}
\def\l{{\lambda}}
\def\de{{\delta}}
\def\be{{\beta}}
\def\ve{{\varepsilon }}

\def\s{{\sigma}}
\def\ep{{\epsilon}}

\begin{document}

\title{Compact complete minimal immersions in $\R^3$}

\author{Antonio Alarc\'{o}n}

\address{Departmento de Geometr\'{i}a y Topolog\'{i}a, Universidad de Granada, E-18071 Granada, Spain}

\curraddr{Departamento de Matem\'{a}tica Aplicada, Universidad de
Murcia, E-30100 Espinardo, Murcia, Spain}

\email{ant.alarcon@um.es}

\thanks{The author was partially supported by Spanish MEC-FEDER Grant MTM2007-61775 and
Regional J. Andaluc\'{i}a Grant P06-FQM-01642.}

\subjclass[2000]{Primary 53A10; Secondary 53C42, 49Q05, 49Q10}

\keywords{Complete minimal surfaces, Plateau problem}

\begin{abstract}
In this paper we find, for any arbitrary finite topological type,
a compact Riemann surface $\mathcal{M},$ an open domain
$M\subset\mathcal{M}$ with the fixed topological type, and a
conformal complete minimal immersion $X:M\to\R^3$ which can be
extended to a continuous map $X:\overline{M}\to\R^3,$ such that
$X_{|\partial M}$ is an embedding and the Hausdorff dimension of
$X(\partial M)$ is $1.$

We also prove that complete minimal surfaces are dense in the
space of minimal surfaces spanning a finite set of closed curves
in $\R^3$, endowed with the topology of the Hausdorff distance.
\end{abstract}

\maketitle

%%%%%%%%%%%%%%%%%%%%%%%%%%%%%%%%%%%%%%%%%%%%%%%%%%%%%%%%%%%%%%%%%%%%%%%%%%%%%%%%%%%%%%%%%%%%%%%%%%%%%%%%%%%%%%%%%%%%%%%%%%%%%%%%%%%%%%%%%%%%%%%%%%%%%%%%
%%%%%%%%%%%%%%%%%%%%%%%%%%%%%%%%%%%%%%%%%%%%%%%%%%%%%%%%%%%%%%%%%%%%%%%%%%%%%%%%%%%%%%%%%%%%%%%%%%%%%%%%%%%%%%%%%%%%%%%%%%%%%%%%%%%%%%%%%%%%%%%%%%%%%%%%
%%%%%%%%%%%%%%%%%%%%%%%%%%%%%%%%%%%%%%%%%%%%%%%%%%%%%%%%%%%%%%%%%%%%%%%%%%%%%%%%%%%%%%%%%%%%%%%%%%%%%%%%%%%%%%%%%%%%%%%%%%%%%%%%%%%%%%%%%%%%%%%%%%%%%%%%

\section{Introduction}\label{sec: intro}

In 1996, Nadirashvili \cite{N1} constructed the first example of a
complete bounded minimally immersed disk in $\R^3,$ giving a
counterexample to the Calabi-Yau conjecture \cite{C}. Calabi
conjectured that a complete minimal surface in $\R^3$ must be
unbounded. Nadirashvili's work seemed to be the end point of a
classical problem. However, his methods and ideas were the
beginning of a significant development in the construction of
complete bounded minimal surfaces. Much work has been done to
study the topology and limit sets of complete bounded minimal
immersions (see for instance \cite{LMM,MM1,MM2}).

Assume $M$ is an open proper subdomain of a compact Riemann
surface and let $X:M\to\R^3$ be a conformal complete minimal
immersion. Denote by $\G$ the limit set of the surface $X(M),$
i.e., $\G\subset\R^3$ is the limit set of $X(p)$ for $p\to\partial
M.$ It was proven by Ferrer, Mart\'{i}n and the author \cite{AFM}
that any open surface of finite topology can be properly and
minimally immersed in any convex domain of $\R^3$ or any bounded
and smooth domain of $\R^3.$ In particular, its limit set is
contained in the boundary of that domain. Moreover, the following
deep result was recently proved by Ferrer, Mart\'{i}n and Meeks
\cite{FMM}: Let $D$ be a domain which is convex (possibly $D =
\R^3$) or smooth and bounded; given any open surface $M,$ there
exists a complete proper minimal immersion $X:M \to D.$

On the other hand, some non-existence theorems imposed some limits
to the theory. Mart\'{i}n, Meeks and Nadirashvili \cite{MMN}
showed the existence of a domain $D\subset\R^3$ for which there is
no complete properly immersed in $D$ minimal surfaces.
Furthermore, there exist requirements for the limit set of compact
complete minimal immersions, according to the following
definition:
\begin{definition}\label{def: compact}
By a compact minimal immersion we mean a minimal immersion
$X:M\to\R^3,$ where $M$ is an open region of a compact Riemann
surface $\mathcal{M},$ and such that $X$ can be extended to a
continuous map $X:\overline{M}\to\R^3.$
\end{definition}
If we take $M=\D,$ the unit disk of $\C,$ and we assume that an
immersion $X:M\to\R^3$ is complete and compact, then its limit set
$\G$ cannot be a subset of the boundary of a cube in $\R^3$
\cite{N2}, i.e., $X$ cannot be proper in a cube. In fact, if such
an immersion is proper in a regular domain $C\subset\R^3,$ then
the second fundamental form of the surface $\partial C$ at any
point of the limit set must be nonnegatively definite \cite{AN}.
Moreover, $C$ cannot be a polyhedron \cite{AN}.

The study of the Calabi-Yau problem gave rise to new lines of work
and techniques. Among other things, these new ideas established a
surprising relationship between the theory of complete minimal
surfaces in $\R^3$ and the Plateau problem. This problem consists
of finding a minimal surface spanning a given closed curve in
$\R^3,$ and it was solved independently by Douglas \cite{D} and
Rad\'{o} \cite{R}, for any Jordan curve. The link between complete
minimal surfaces and the Plateau problem is the existence of
compact complete minimal immersions in $\R^3.$ Mart\'{i}n and
Nadirashvili \cite{MN} constructed compact complete conformal
minimal immersions $X:\D\to\R^3$ such that $X_{|\partial \D}$ is
an embedding and $X(\esf^1)$ is a Jordan curve with Hausdorff
dimension $1.$ Furthermore, they showed that the set of Jordan
curves $X(\esf^1)$ constructed by the above procedure is dense in
the space of Jordan curves of $\R^3$ with the Hausdorff distance.

The aim of this paper is to join the techniques used in the
construction of complete bounded minimal surfaces with arbitrary
finite topology, and those used to construct complete minimal
disks spanning Jordan curves, in order to prove the following
result.

\begin{theoremintro}
There exist compact complete minimal immersions $X:M\to\R^3$ of
arbitrary finite topological type and such that the extended map
$X:\overline{M}\to\R^3$ satisfies that $X_{|\partial M}$ is an
embedding and the Hausdorff dimension of $X(\partial M)$ is $1.$

Moreover, for any finite family $\Sigma$ of closed curves in
$\R^3$ which admits a solution to the Plateau problem, and for any
$\xi>0,$ there exists a minimal immersion $X:M\to\R^3$ in the
above conditions and such that $\de^H(\Sigma,X(\partial M))<\xi,$
where $\de^H$ means the Hausdorff distance.
\end{theoremintro}

Unfortunately, our arguments do not guarantee that the boundary
curves of $M$ are Jordan curves in $\mathcal{M}.$ Therefore,
although $X_{|\partial M}$ is an embedding, we cannot assert that
$X(\partial M)$ consists of Jordan curves. It is also important to
notice that $X$ cannot be an embedding following the recent
result by Colding and Minicozzi \cite{CM} which asserts that any
complete embedded minimal surface in $\R^3$ with finite topology
is proper in $\R^3.$ In particular, such a surface must be
unbounded, and so it cannot be contained in a compact set.

In this paper we also prove the following density result for
compact complete minimal immersions in $\R^3.$

\begin{theoremintro}
Complete minimal surfaces are dense in the space of minimal
surfaces spanning a finite set of closed curves in $\R^3$, endowed
with the topology of the Hausdorff distance.
\end{theoremintro}

\noindent{\bf Acknowledgments.} The author is indebted to
Francisco Mart\'{i}n for some informative conversations on the
paper \cite{MN}.

%%%%%%%%%%%%%%%%%%%%%%%%%%%%%%%%%%%%%%%%%%%%%%%%%%%%%%%%%%%%%%%%%%%%%%%%%%%%%%%%%%%%%%%%%%%%%%%%%%%%%%%%%%%%%%%%%%%%%%%%%%%%%%%%%%%%%%%%%%%%%%%%%%%%%%%%
%%%%%%%%%%%%%%%%%%%%%%%%%%%%%%%%%%%%%%%%%%%%%%%%%%%%%%%%%%%%%%%%%%%%%%%%%%%%%%%%%%%%%%%%%%%%%%%%%%%%%%%%%%%%%%%%%%%%%%%%%%%%%%%%%%%%%%%%%%%%%%%%%%%%%%%%
%%%%%%%%%%%%%%%%%%%%%%%%%%%%%%%%%%%%%%%%%%%%%%%%%%%%%%%%%%%%%%%%%%%%%%%%%%%%%%%%%%%%%%%%%%%%%%%%%%%%%%%%%%%%%%%%%%%%%%%%%%%%%%%%%%%%%%%%%%%%%%%%%%%%%%%%

\section{Preliminaries}\label{sec: preli}

This section is devoted to briefly summarize the notation and
results that we use in the paper.

\subsection{Riemann surfaces background}

Throughout the paper we work on a compact Riemann surface endowed
with a Riemannian metric. We consider that the following data are
fixed.

\begin{definition}\label{def: M-ds}
Let $M'$ be a compact Riemann surface of genus $\s\in\N \cup \{0\},$ and $ds^2$ a Riemannian metric in $M'.$
\end{definition}

 Consider a subset $W \subset M',$ and a Riemannian metric $d\tau^2$ in $W.$ Given a curve $\a$ in $W,$ by $\longui_{d\tau}(\a)$ we mean the length of $\a$ with respect to the metric $d\tau^2.$ Moreover, we define:
\begin{enumerate}[1]
\item[$\bullet$] $\dist_{(W,d\tau)}(p,q)=\inf \{\longui_{d\tau}(\alpha) \: | \: \alpha:[0,1]\rightarrow W, \; \alpha(0)=p,\alpha(1)=q \}$, for any $p,q\in W.$
\item[$\bullet$] $\dist_{(W,d\tau)}(T_1,T_2)=\inf \{\dist_{(W,d\tau)}(p,q) \;|\;p \in T_1, \;q \in T_2 \}$, for any $T_1, T_2 \subset  W.$
\item[$\bullet$] $\diam_{d\tau}(W)=\sup\{\dist_{(W,d\tau)}(p,q)\;|\;p,q\in W\}.$
\end{enumerate}
Throughout the paper, we use to work with a domain $W$ in $M'$ and
a conformal minimal immersion $Y:\overline{W}\to\R^3.$ Then, by
$ds_{Y}^2$ we mean the Riemannian metric induced by $Y$ in
$\overline{W}.$ We also write $\dist_{(\overline{W},Y)}(T_1,T_2)$
instead of $\dist_{(\overline{W},ds_Y)}(T_1,T_2),$ for any sets
$T_1$ and $T_2$ in $\overline{W}.$

For $\textsc{e}\in\N$, consider $\D_1,\ldots,\D_\textsc{e}\subset M'$ open
disks so that $\{\g_i:=\partial \D_i\}_{i=1}^{\textsc{e}}$ are analytic Jordan
curves and $\overline{\D}_i\cap \overline{\D}_j=\emptyset$ for all $i\neq j$.

\begin{definition}\label{multicycle}
Each curve $\g_i$ will be called a cycle on $M'$ and the family
$\mathcal{ J}=\{\g_1,\ldots,\g_\textsc{e}\}$ will be called a multicycle on $M'$.
We denote by $\intc(\g_i)$ the disk $\D_i$, for $i=1,\ldots, \textsc{e}.$
We also define $M(\mathcal{ J})=M'\setminus(\cup_{i=1}^\textsc{e} \overline{\intc(\g_i)})$.
\end{definition}
Given $\mathcal{ J}=\{\g_1,\ldots,\g_\textsc{e}\}$ and $\mathcal{
J}'=\{\g_1',\ldots,\g_\textsc{e}'\}$ two multicycles on $M'$ we
write $\mathcal{ J}'< \mathcal{ J}$ if $\overline{\intc(\g_i)}
\subset \intc (\g_i')$ for $i=1,\ldots, \textsc{e}.$ Notice that
$\mathcal{ J}'< \mathcal{ J}$ implies $\overline{M(\mathcal{ J}')}
\subset M(\mathcal{ J})$.

Let $\mathcal{ J}=\{\g_1,\ldots,\g_\textsc{e}\}$ be a multicycle
on $M'.$ For $\ep>0$ small enough, we can consider the multicycle
$\mathcal{J}^\ep=\{\g_1^\ep,\ldots,\g_\textsc{e}^\ep\},$ where by
$\g_i^\ep$ we mean the cycle satisfying $\overline{\intc
(\g_i)}\subset\intc(\g_i^\ep)$ and $\dist_{(M',ds)}(q,\g_i)=\ep$
for all $q\in \g_i^\ep$ and $i=1,\ldots, \textsc{e}.$ Notice that
$\J^\ep<\J.$

\begin{definition}\label{def: infimo}
Let $K$ be a compact subset in the Riemann surface $M',$ $n\in\N,$
and $f:K\to\R^3$ a continuous embedding. We define
\[
\Psi(K,f,n):=\frac1{2n^2} \inf \left\{ \|f(p)-f(q)\|\;\left|\;
p\neq q\in K\;,\; \dist_{(M',ds)}(p,q)>\frac1{n}\right\}\right..
\]
Notice that $\Psi(K,f,n)>0.$
\end{definition}

%%%%%%%%%%%%%%%%%%%%%%%%%%%%%%%%%%%%%%%%%%%%%%%%%%%%%%%%%%%%%%%%%%%%%%%%%%%%%%%%%%%%%%%%%%%%%%%%%%%%%%%%%%%%%%%%%%%%%%%%%%%%%%%%%%%%%%%%%%%%%%%%%%%%%%%%
%%%%%%%%%%%%%%%%%%%%%%%%%%%%%%%%%%%%%%%%%%%%%%%%%%%%%%%%%%%%%%%%%%%%%%%%%%%%%%%%%%%%%%%%%%%%%%%%%%%%%%%%%%%%%%%%%%%%%%%%%%%%%%%%%%%%%%%%%%%%%%%%%%%%%%%%
%%%%%%%%%%%%%%%%%%%%%%%%%%%%%%%%%%%%%%%%%%%%%%%%%%%%%%%%%%%%%%%%%%%%%%%%%%%%%%%%%%%%%%%%%%%%%%%%%%%%%%%%%%%%%%%%%%%%%%%%%%%%%%%%%%%%%%%%%%%%%%%%%%%%%%%%

\subsection{Minimal surfaces background}

The theory of complete minimal surfaces is closely related to the theory of Riemann surfaces. This is due to the fact that any such surface is given by a triple $\Phi=(\Phi_1, \Phi_2, \Phi_3)$ of holomorphic 1-forms defined on some Riemann surface such that
\begin{equation} \label{eq:conforme}
\Phi_1^2+\Phi_2^2+\Phi_3^2=0,
\end{equation}
\[
\|\Phi_1\|^2+\|\Phi_2\|^2+\|\Phi_3\|^2 \neq 0,
\]
and all periods of the $\Phi_j$ are purely imaginary. Here we
consider $\Phi_i$ to be a holomorphic function times $dz$ in a
local parameter $z$. Then the minimal immersion $X:M \rightarrow
\R^3$ can be parameterized by $z \mapsto \mbox{Re} \int^z \Phi.$
The above triple is called the Weierstrass representation of the
immersion $X$. Usually, the first requirement \eqref{eq:conforme}
(which ensures the conformality of $X$) is guaranteed by
introducing the formulas
\[
\Phi_1 =\frac12 \left( 1-g^2\right) \, \eta, \quad \Phi_2
=\frac{\rm i}2 \left( 1+g^2\right) \, \eta, \quad \Phi_3= g \,
\eta,
\]
with a meromorphic function $g$ (the stereographic projection of the Gauss map) and a holomorphic 1-form $\eta$. The metric of $X$ can be expressed as
\begin{equation}\label{eq:metric}
ds_X^2=\tfrac12 \|\Phi\|^2=\left(\tfrac12\left(1+|g|^2\right) \|\eta \|\right)^2.
\end{equation}

\subsubsection{The L\'{o}pez-Ros transformation}
The proof of Lemma \ref{lem: gafa-3} exploits what has come to be
called the L\'{o}pez-Ros transformation. If $M$ is a Riemann
surface and $(g,\eta)$ are the Weierstrass data of a minimal
immersion $X:M \rightarrow \R^3$, we define on $M$ the data
\[
\widetilde g= \frac{g}{h}, \qquad \widetilde \eta= \eta \cdot h,
\]
where $h:M \rightarrow \C$ is a holomorphic function without
zeros. If the periods of this new Weierstrass representation are
purely imaginary, then it defines a minimal immersion $\widetilde
X: M \rightarrow \R^3$. This method provides us with a powerful
and natural tool for deforming minimal surfaces. From our point of
view, the most important property of the resulting surface is that
the third coordinate function is preserved. Note that the
intrinsic metric is given by \eqref{eq:metric} as
\[
ds_{\widetilde
X}^2=\left(\frac12\left(|h|+\frac{|g|^2}{|h|}\right)\, \|\eta \|
\right)^2.
\]
This means that we can increase the intrinsic distance in a prescribed compact of $M$, by using  suitable functions $h$.

%%%%%%%%%%%%%%%%%%%%%%%%%%%%%%%%%%%%%%%%%%%%%%%%%%%%%%%%%%%%%%%%%%%%%%%%%%%%%%%%%%%%%%%%%%%%%%%%%%%%%%%%%%%%%%%%%%%%%%%%%%%%%%%%%%%%%%%%%%%%%%%%%%%%%%%%
%%%%%%%%%%%%%%%%%%%%%%%%%%%%%%%%%%%%%%%%%%%%%%%%%%%%%%%%%%%%%%%%%%%%%%%%%%%%%%%%%%%%%%%%%%%%%%%%%%%%%%%%%%%%%%%%%%%%%%%%%%%%%%%%%%%%%%%%%%%%%%%%%%%%%%%%
%%%%%%%%%%%%%%%%%%%%%%%%%%%%%%%%%%%%%%%%%%%%%%%%%%%%%%%%%%%%%%%%%%%%%%%%%%%%%%%%%%%%%%%%%%%%%%%%%%%%%%%%%%%%%%%%%%%%%%%%%%%%%%%%%%%%%%%%%%%%%%%%%%%%%%%%

\subsection{Hausdorff distance}

Throughout the paper we denote by $\de^H(C,D)$ the Hausdorff distance
between two nonempty compact subsets of $\R^n,$ $C$ and $D,$ i.e.,
\[
\de^H(C,D)=\max\left\{ \sup_{x\in C} \inf_{y\in D} \|x-y\|\;,\;
\sup_{y\in D} \inf_{x\in X} \|x-y\|\right\}.
\]

\subsection{Preliminary lemma}

Consider a multicycle $\I=\{\Sigma_1,\ldots,\Sigma_\ee\}$ on the
compact surface $M'.$ We denote $M=\overline{M(\I)}.$ For any
$i\in\{1,\ldots,\ee\},$ let $\be_i:\Sigma_i\to\R^3$ be an analytic
curve. Finally, given $\mathcal{T}(\Sigma_i)$ a
tubular neighborhood of $\Sigma_i$ in $M,$ %$(M,ds),$
we denote by $\pp_i:\mathcal{T}(\Sigma_i)\to\Sigma_i$ the natural
projection.

Taking this set up into account, the following result was
essentially proved by Ferrer, Mart\'{i}n and the author
\cite[Lemma 3]{AFM}. It concentrates most of the technical
computations required in the main lemma of this paper (Lemma
\ref{lem: gafa-4}).

\begin{lemma}\label{lem: gafa-3}
Consider $\J=\{\g_1,\ldots,\g_\ee\}$ a multicycle on $M,$
$X:\overline{M(\J)}\to\R^3$ a conformal minimal immersion, $p_0$ a
point in $M(\J),$ and $r>0,$ such that
\begin{enumerate}[1]
\item[$\bullet$] $X(p_0)=0.$
\item[$\bullet$] $\g_i\subset \mathcal{T}(\Sigma_i),$ for any $i\in\{1,\ldots,\ee\}.$
\item[$\bullet$] $\|X(p)-\be_i(\pp_i(p))\|<r,$ for all $p\in \g_i$ and for all $i=1,\ldots,\ee.$
\end{enumerate}
Then, for any $s>0,$ and any $\ep>0$ so that $p_0\in M(\J^\ep),$
there exist a multicycle
$\widetilde{\J}=\{\widetilde{\g}_1,\ldots,\widetilde{\g}_\ee\}$
and a conformal minimal immersion
$\widetilde{X}:\overline{M(\widetilde{\J})}\to\R^3$ satisfying
\begin{enumerate}[\rm ({L\ref{lem: gafa-3}}.1)]
\item $\til{X}(p_0)=0.$

\item $\til{X}(\til{\g}_i)$ is a Jordan curve, $\forall
i=1,\ldots,\ee.$

\item $\til{\g}_i\subset \mathcal{T}(\Sigma_i),$ for all $i=1,\ldots,\ee.$

\item $\J^\ep<\til{\J}<\J.$

\item $s<\dist_{(\overline{M(\til{\J})},\til{X})}(p,\til{\J}),$ $\forall p\in \J^\ep.$

\item $\|\til{X}-X\|<\ep$ in $\overline{M(\J^\ep)}.$

\item $\|\til{X}(p)-\be_i(\pp_i(p))\|<R=\sqrt{4s^2+r^2}+\ep,$ $\forall p\in\til{\g}_i,$ $\forall i=1,\ldots,\ee.$
\end{enumerate}
\end{lemma}

Actually, the above lemma improves Lemma 3 in \cite{AFM}. The
improvement consists of the addition of Statement ({L\ref{lem:
gafa-3}}.2). Checking the proof given in \cite{AFM}, we notice
that this assertion can be trivially guaranteed. Item ({L\ref{lem:
gafa-3}}.2) holds up to infinitesimal variations of $\til{\J}.$
These variations can be done without loosing the remaining
properties.

\subsubsection{Sketch of the proof of Lemma \ref{lem: gafa-3}}

The proof is inspired by a technique introduced by Mart\'{i}n and
Nadirashvili in \cite{MN}. See \cite[\S 5.1]{AFM} for further
details.

First of all, fix $\omega$ a meromorphic differential on $M'$ so
that $\omega$  has neither zeroes nor poles on $\overline{M(\J)}.$
Then it is well known that $\|\omega\|^2$ is a flat Riemannian
metric on $\overline{M(\J)}.$ Assume $ds^2=\|\omega\|^2.$

Let $\zeta\in]0,\ep[$ be small enough so that $\g_i^{\zeta}\subset \mathcal{T}(\Sigma_i),$ for $i=1,\ldots,\ee.$ Consider $N\in\N$ such that $2/N<\zeta.$

Throughout the sketch of the proof, a set of real positive constants not depending on $N$ will appear. The symbol `$\textrm{const}$' will denote these different constants.

We consider an order relation in the set $I\equiv
\{1,\ldots,2N\}\times\{1,\ldots,\ee\}.$ We say $(j,l)>(i,k)$ if
one of the following situations holds: $l=k$ and $j>i,$ or $l>k.$

The next step consists of the construction of a Nadirashvili's type labyrinth along the cycle $\g_k,$ $k=1,\ldots,\ee.$ Let $\{v_{1,k}, \ldots ,v_{2N,k}\}$ be a set of points in the curve $\g_k$ that divide $\g_k$ into $2N$ equal parts (i.e., curves with the same length). Following the normal projection, we can transfer the above partition to the curve $\g_k^{2/N}$: $\{v_{1,k}^\prime, \ldots, v_{2N,k}'\}$. We define the following sets:

\begin{enumerate}[1]
\item[$\bullet$] $L_{i,k}$ is the minimizing geodesic in $(\overline{M(\J)},ds^2)$ joining $v_{i,k}$ and
$v_{i,k}',$ $(i,k)\in I.$

\item[$\bullet$] $\mathcal{ G}_{j,k}=\g_k^{j/N^3},$ $\forall\, j=0,\ldots, 2N^2$ (recall that $\g_k^{j/N^3}$ means the parallel curve to $\g_k,$ in $M(\J),$ such that the distance between them is $j/N^3$).

\item[$\bullet$] $\mathcal{ A}_k=\bigcup_{j=0}^{N^2-1}\overline{\intc{\mathcal{ G}_{2j+1,k}}\setminus \intc{\mathcal{ G}_{2j,k}}}$ and $\widetilde{\mathcal{ A}}_k=\bigcup_{j=1}^{N^2}\overline{\intc{\mathcal{ G}_{2j,k}}\setminus \intc{\mathcal{ G}_{2j-1,k}}}.$

\item[$\bullet$] $\mathcal{ R}_k= \bigcup_{j=0}^{2N^2} \mathcal{ G}_{j,k},$ $\mathcal{ B}_k= \bigcup_{j=1}^N L_{2j,k},$ and $\widetilde{\mathcal{ B}}_k= \bigcup_{j=0}^{N-1} L_{2j+1,k}.$

\item[$\bullet$] $\mathcal{ L}_k=\mathcal{B}_k \cap \mathcal{A}_k,$ $\mathcal{ \widetilde{L}}_k=\widetilde{\mathcal{B}}_k \cap \widetilde{\mathcal{A}}_k,$ and $H_k=\mathcal{R}_k \cup \mathcal{ L}_k \cup \mathcal{ \widetilde{L}}_k.$

\item[$\bullet$] $\Omega_{N,k}=\{ p \in \intc{(\mathcal{ G}_{2N^2,k})} \setminus \intc{(\mathcal{ G}_{0,k})}\;|\; \dist_{(M,ds)}(p,H_k) \geq {1}/{4N^3}\}.$

\item[$\bullet$] $\Omega_N=\bigcup_{k=1}^\textsc{e} \Omega_{N,k}.$

\item[$\bullet$] $\omega_i^k$ is the union of the curve $L_{i,k}$ and those connected components of $\Omega_{N,k}$ that have nonempty intersection with $L_{i,k}$ for $(i,k)\in I.$

\item[$\bullet$] $\varpi_i^k = \{ p \in M \; | \; \dist_{(M,ds)}(p,\omega_i^k)< \de(N) \},$ where  $\de(N)>0$ is chosen in such a way that the sets $\overline{\varpi}_i^k,$ $(i,k)\in I,$ are pairwise disjoint.
\end{enumerate}

Now, assume that $N$ is large enough so that the following
assertions hold:
\begin{enumerate}[\rm ({A}1)]
\item ${\rm diam}_{(M,ds)}(\varpi_i^k)<{\cte}/{N}.$

\item If $\l^2\cdot ds^2$ is a conformal metric on $\overline{M(\J)}$ that satisfies
\[
\l\geq
\begin{cases}
c & \text{ in } M(\mathcal{ J})\\
c\; N^4 & \text{ in } \Omega_N,
\end{cases}
\]
for $c>0$, and if $\a$ is a curve in $\overline{M(\mathcal{ J})}$
connecting $\gamma_k^{\zeta}$ and $\gamma_k$, for some $k \in \{1,
\ldots,\ee\}$, then $\longui_{\l\cdot ds}(\a)\geq \cte c \; N$.
\end{enumerate}

The next step consists of the construction, for $N$ large enough,
of a sequence $\{F_i^k\;|\;(i,k)\in I\}$ of conformal minimal
immersions (with boundary) defined on $\overline{M(\mathcal{
J})}.$ We consider $F_0^1=X$ and denote $F_0^k=F_{2N}^{k-1},$
$\forall k=2,\ldots,\textsc{e}.$ These immersions must be
constructed to satisfy the following list of properties:
\begin{enumerate}[\rm (1$_{i}^{k}$)]
\item $F_i^k(p)= {\rm Re}\, \left( \int_{p_0}^p \Phi^{i,k}\right),$ where $\Phi^{i,k}=\phi^{i,k}\,\omega.$

\item $\|\phi^{i,k}(p)-\phi^{i-1,k}(p) \| \leq 1/N^2,$ for all $p \in \overline{M(\mathcal{ J})}\setminus \varpi_i^k.$

\item $\|\phi^{i,k}(p) \| \geq N^{7/2},$ for all $p \in \omega_i^k.$

\item $\| \phi^{i,k}(p) \| \geq {\cte}/{\sqrt{N}},$ for all $p \in \varpi_i^k.$

\item $\dist_{\esf^2}(G_i^k(p), G_{i-1}^k(p))<{1}/{N^2},$ for all $p \in \overline{M(\mathcal{ J})}\setminus\varpi_i^k,$ where $\dist_{\esf^2}$ is the intrinsic distance in $\esf^2$ and $G_i^k$ represents the Gauss map of $F_i^k.$

\item There exists an orthonormal  basis of $\R^3,$ $S_i^k=\{ e_1, e_2, e_3 \}$ such that\footnote{Given $v \in \R^3,$ and $S$ a coordinate system, we let $v_{(k,S)}$ denote the $k$-th coordinate of $v$ in $S$. The first two coordinates of $v$ in this basis are represented by $v_{(*,S)}=(v_{(1,S)},v_{(2, S)})$.}
\begin{enumerate}[\rm ({6}.1$_{i}^{k}$)]
\item For any $p \in \overline{\varpi_i^k}$ with $\|X(p)-\be_k({\tt P}_k(p))\|\geq 1/\sqrt{N},$ we have $\| (X(p)-\be_k({\tt P}_k(p)))_{(*,S_i^k)} \| <{\cte}/{\sqrt{N}}.$
\item $(F_{i}^k(p))_{(3,S_i^k)}=(F_{i-1}^k(p))_{(3,S_i^k)},$ for all $p \in \overline{M(\mathcal{ J})}.$
\end{enumerate}

\item $\| F_i^k(p)-F_{i-1}^k(p) \| \leq {\cte}/{N^2},$ $\forall\, p \in \overline{M(\mathcal{ J})} \setminus\varpi_i^k.$
\end{enumerate}
The sequence $\{F_i^k\;|\;(i,k)\in I\}$ is constructed in a
recursive way. We construct $F_i^k$ starting from $F_{i-1}^k$ by
using a L\'{o}pez-Ros transformation. When $i-1=0$ we adopt the
convention that $F_0^k:=F_{2N}^{k-1},$ if $k>1,$ and $F_0^1:=X$.
The same holds for the Weierstrass representations.

Suppose that we have $\{F_j^l\;|\;(j,l)<(i,k)\}.$ Assume the
existence of the frame $S_i^k$ satisfying (6.1$_i^k$). It is
implied by Properties (A1) and (5$_j^l$), $(j,l)<(i,k)$ (we omit
the details on how to obtain it). Let $(g^{i-1,k},\Phi_3^{i-1,k})$
be the Weierstrass data of $F_{i-1}^k$ in $S_i^k.$ To construct
the immersion $F_i^k,$ we consider a positive $\a,$ and a
holomorphic function $h_\a:\overline{M(\J)}\to\C^*$ satisfying
\begin{enumerate}[1]
\item[$\bullet$] $|h_\a-\a|<1/\a,$ in $\omega_i^k.$
\item[$\bullet$] $|h_\a-1|<1/\a,$ in $\overline{M(\mathcal{ J})}\setminus \varpi_i^k.$
\item[$\bullet$] The Weierstrass data
$(g^{i,k}:=g^{i-1,k}/h_\a, \Phi_3^{i,k}:=\Phi_3^{i-1,k})$ give a
well-defined minimal immersion in $\overline{M(\mathcal{ J})}.$
\end{enumerate}
The existence of such a function is known \cite[Lemma 1]{AFM}.
Then, for a large enough value of the parameter $\a,$ the
immersion $F_i^k(p):=\mbox{Re}\int_{p_0}^p\Phi^{i,k}$ whose
Weierstrass data in $S_i^k$ are $(g^{i,k},\Phi_3^{i,k})$ satisfies
the above list of properties.  In this way, we have constructed
the sequence $\{F_i^k\;|\;(i,k)\in I\}$ satisfying Properties
(1$_i^k$)$,\ldots,$(7$_i^k$).

If $N$ is large enough, then the set
\[
\mathcal{D}=\{p\in\overline{M(\J)}\setminus M(\J^{\zeta})\;|\; s<\dist_{(\overline{M(\J)},F_{2N}^\ee)}(p,M(\J^{\zeta}))<2s\}
\]
is nonempty, and $\J$ and $\J^\ep$ are contained in different
connected components of $\overline{M(\J)}\setminus \mathcal{D}.$
To check this claim, use Properties (A2), (2$_i^k$), (3$_i^k$),
(4$_i^k$), $(i,k)\in I,$ and take \eqref{eq:metric} into account.
Then, for any multicycle $\widetilde{\J}$ on $\mathcal{D},$ the
immersion
$\widetilde{X}:=F_{2N}^\ee:\overline{M(\widetilde{\J})}\to\R^3$
satisfies the conclusion of Lemma \ref{lem: gafa-3} except for
({L\ref{lem: gafa-3}}.2). Indeed, ({L\ref{lem: gafa-3}}.1) is
implied by (1$_{2N}^\ee$). Statement ({L\ref{lem: gafa-3}}.6)
trivially follows from (7$_i^k$), $(i,k)\in I.$ From (6$_i^k$),
(7$_i^k$), $(i,k)\in I,$ and the choice of $\widetilde{\J},$ we
obtain, in a more elaborated way (we omit the details),
({L\ref{lem: gafa-3}}.7). The remainder properties follow from the
choice of $\widetilde{\J}$ and the positive $\zeta.$ Furthermore,
it is clear that, up to infinitesimal variations of the cycle
$\til{\g}_i\subset\mathcal{D},$ it can be guaranteed that
$\til{X}(\til{\g}_i)$ has no self-intersections, $\forall
i=1,\ldots,\ee,$ which proves ({L\ref{lem: gafa-3}}.2).

%%%%%%%%%%%%%%%%%%%%%%%%%%%%%%%%%%%%%%%%%%%%%%%%%%%%%%%%%%%%%%%%%%%%%%%%%%%%%%%%%%%%%%%%%%%%%%%%%%%%%%%%%%%%%%%%%%%%%%%%%%%%%%%%%%%%%%%%%%%%%%%%%%%%%%%%
%%%%%%%%%%%%%%%%%%%%%%%%%%%%%%%%%%%%%%%%%%%%%%%%%%%%%%%%%%%%%%%%%%%%%%%%%%%%%%%%%%%%%%%%%%%%%%%%%%%%%%%%%%%%%%%%%%%%%%%%%%%%%%%%%%%%%%%%%%%%%%%%%%%%%%%%
%%%%%%%%%%%%%%%%%%%%%%%%%%%%%%%%%%%%%%%%%%%%%%%%%%%%%%%%%%%%%%%%%%%%%%%%%%%%%%%%%%%%%%%%%%%%%%%%%%%%%%%%%%%%%%%%%%%%%%%%%%%%%%%%%%%%%%%%%%%%%%%%%%%%%%%%

\section{Main lemma}

The next result is the kernel of this paper. It shows that it is
possible to perturb a given minimal surface with boundary and
finite topology outside a compact set in such a way that the
intrinsic diameter of the surface can be increased, while the
extrinsic one is preserved. At the same time, it can be obtained
the embeddedness of the resulting immersion along its boundary
curves.

\begin{lemma}\label{lem: gafa-4}
Let $\J=\{\g_1,\ldots,\g_\ee\}$ be a multicycle on $M',$
$X:\overline{M(\J)}\to\R^3$ a conformal minimal immersion, and
$p_0\in M(\J)$ with $X(p_0)=0.$

Then, for any $\l>0,$ and any $\mu>0$ such that $p_0\in
M(\J^\mu),$ there exist a multicycle
$\gorro{\J}=\{\gorro{\g}_1,\ldots,\gorro{\g}_\ee\}$ and a
conformal minimal immersion
$\gorro{X}:\overline{M(\gorro{\J})}\to\R^3$ with the following
properties:
\begin{enumerate}[\rm ({L\ref{lem: gafa-4}}.1)]
\item $\gorro{X}(p_0)=0.$

\item $\gorro{X}(\gorro{\g}_i)$ is a Jordan curve, for any $i=1,\ldots,\ee.$

\item $\gorro{\g}_i\subset \mathcal{T}(\g_i),$ where $\mathcal{T}(\g_i)$ is a tubular neighborhood of $\g_i,$
for $i=1,\ldots,\ee.$

\item $\J^\mu<\gorro{\J}<\J.$

\item $\l<\dist_{(\overline{M(\gorro{\J})},\gorro{X})}(p,\gorro{\J}),$ $\forall p\in \J^\mu.$

\item $\|\gorro{X}-X\|<\mu$ in $\overline{M(\gorro{\J})}.$

%\item $\|\gorro{X}(p)-X(\pp_i(p))\|<\mu,$ for all $p\in \gorro{\g}_i,$ where $\pp_i:\mathcal{T}(\g_i)\to \g_i$ is the natural projection, for any $i=1,\ldots,\ee.$

\item $\de^H\big( X(\overline{M(\J)}),
\gorro{X}(\overline{M(\gorro{\J})})\big)<\mu.$
\end{enumerate}
\end{lemma}

\begin{proof}
Consider $c_0<\mu,$ $r_1$ and $\rho_1$ three positive constants which will be specified later. Define
\[
r_n=\sqrt{r_{n-1}^2+\left( \frac{2c_0}{n}
\right)^2}+\frac{c_0}{n^2},\quad \rho_n=\rho_1+\sum_{i=2}^n
\frac{c_0}{i}.
\]
Assume the constants $c_0$ and $r_1$ are small enough so
that
\begin{equation}\label{ecu: gafa-3.6}
\lim_{n\to\infty} r_n<\frac{\mu}2.
\end{equation}

The next step consists of applying Lemma \ref{lem: gafa-3}. In
order to do it, we need a family $\Sigma_1,\ldots,\Sigma_\ee,$ of
analytic cycles on $M'.$ We consider $\Sigma_i:=\g_i,$ $\forall
i=1,\ldots,\ee.$ Hence it is clear that
$\mathcal{I}=\{\Sigma_1,\ldots,\Sigma_\ee\}$ is a multicycle on
$M'.$ We also need a family of analytic curves
$\be_i:\Sigma_i\to\R^3.$ Take $\be_i:=X_{|\g_i},$ $\forall
i=1,\ldots,\ee.$

For any natural $n$ let us construct a family
$\chi_n=\{\J_n,X_n,\ep_n\}$ consisting of
\begin{enumerate}
\item[\rm (1$_{n}$)] $\J_n=\{\g_{n,1},\ldots,\g_{n,\ee}\}$ is a multicycle with $\g_{n,i}\subset \mathcal{T}(\Sigma_i)$ for $i=1,\ldots,\ee.$
\item[\rm (2$_{n}$)] $X_n:\overline{M(\J_n)}\to \R^3$ is a conformal minimal immersion with
$X_n(p_0)=0,$ and such that $X_n(\g_{n,i})$ is a Jordan curve, for
any $i=1,\ldots,\ee$ (only holds for $n\geq 2$).
\item[\rm (3$_{n}$)] $\{\ep_n\}_{n\in\N}$ is a decreasing sequence
of positives such that $\ep_n<c_0/n^2.$
\end{enumerate}
Moreover, the sequence $\{\chi_n\}_{n\in\N}$ will satisfy the
following properties:
\begin{enumerate}
\item[\rm (4$_{n}$)] $\J^\mu<\J_{n-1}^{\ep_n}<\J_n<\J_{n-1}$ (only holds for $n\geq 2$).
\item[\rm (5$_{n}$)]
$\rho_n<\dist_{(\overline{M(\J_n)},X_n)}(p,\J_n),$ for all $p\in
\J^\mu.$
\item[\rm (6$_{n}$)] $\|X_n(p)-X(\pp_i(p))\|<r_n,$ for any $p\in \g_{n,i},$ for any
$i=1,\ldots,\ee$ (recall that
$\pp_i:\mathcal{T}(\Sigma_i)\to\Sigma_i$ is the natural
projection).
\end{enumerate}

The sequence will be obtained following an inductive method. For
the first family, $\chi_1,$ we choose $X_1=X$ and $\J_1=\J.$ This
choice and the hypotheses of the lemma guarantee that Properties
(1$_1$) and (6$_1$) are satisfied. Notice that Property (2$_1$)
might not hold. This is not a problem, it will not be used in the
construction of the remainder elements of the sequence
$\{\chi_n\}_{n\geq 2}.$ Let $\rho_1$ be small enough such that
\[
\rho_1<\dist_{(\overline{M(\J_1)},X_1)}(p,\J_1),\quad \forall p\in
\J^\mu.
\]
This choice implies (5$_1$). Finally, we take
$\ep_1<\min\{c_0,r_1\}.$ Hence, (3$_1$) holds. The family $\chi_1$
is well defined. Moreover, we assume that $\ep_1$ is small enough
so that $\g_i^{\ep_1}\subset \mathcal{T}(\Sigma_i),$ $\forall
i=1,\ldots,\ee,$ and so that given $p$ in the connected component
of $\overline{M(\J)}\setminus M(\J^{\ep_1})$ around $\g_i,$ for
$i=1,\ldots,\ee,$ then
\begin{equation}\label{ecu: gafa-5.7}
\|X(p)-X(\pp_i(p))\|<r_1<\frac{\mu}2.
\end{equation}
This choice is possible since $\J_1=\J$ and $X$ is uniformly
continuous.

Assume that we have defined $\chi_1,\ldots, \chi_{n-1}.$ In order
to construct $\chi_n$ we take a real number $\xi>0$ such that
\begin{equation}\label{ecu: remedios}
\rho_{n-1}<\dist_{(\overline{M(\J_{n-1})},X_{n-1})}(p,\J_{n-1}^{\xi}),\quad
\forall p\in \J^\mu.
\end{equation}
Its existence is implied by  (5$_{n-1}$). Consider
$\ep_n<\min\{\ep_{n-1},c_0/n^2,\xi\}.$ Hence (3$_n$) holds. Now,
we consider the multicycle $\J_n$ and the conformal minimal
immersion $X_n:\overline{M(\J_n)}\to\R^3$ given by Lemma \ref{lem:
gafa-3} for the data
\[
X=X_{n-1},\quad \J=\J_{n-1},\quad r=r_{n-1},\quad
s=\frac{c_0}{n},\quad \ep=\ep_n.
\]
Recall that $\Sigma_i=\g_i$ and $\be_i=X_{|\g_i},$
$i=1,\ldots,\ee.$ Then Properties ({L\ref{lem: gafa-3}}.3), and
({L\ref{lem: gafa-3}}.1) and ({L\ref{lem: gafa-3}}.2) in Lemma
\ref{lem: gafa-3} imply (1$_n$) and (2$_n$), respectively. (4$_n$)
follows from ({L\ref{lem: gafa-3}}.4) and the fact that
$\sum_{k=2}^n\ep_k<\mu.$ In order to check (5$_n$) we assume that
$\ep_n$ was chosen small enough so that
\[
\rho_{n-1}<\dist_{(\overline{M(\J_{n})},X_{n})}(p,\J_{n-1}^{\xi}),\quad
\forall p\in \J^\mu.
\]
This assumption is possible thanks to equation \eqref{ecu:
remedios}, and since ({L\ref{lem: gafa-3}}.6) guarantees that
$X_n$ converges to $X_{n-1}$ uniformly on
$\overline{M(\J_{n-1}^{\xi})},$ as $\ep_n$ goes to zero. From this
inequality, the definition of $\rho_n$ and ({L\ref{lem:
gafa-3}}.5) we conclude (5$_n$). Finally, ({L\ref{lem: gafa-3}}.7)
trivially implies (6$_n$). In this way, we have constructed the
sequence $\{\chi_n\}_{n\in\N}$ satisfying Properties
(1$_n$),$\ldots,$ (6$_n$).

Let us check that, for $n_0$ large enough, the multicycle
$\J_{n_0}$ and the immersion $X_{n_0}$ satisfy the conclusion of
Lemma \ref{lem: gafa-4}. Indeed, Properties (5$_n$), $n\in\N,$ and
the fact that the sequence $\{\rho_n\}_{n\in\N}$ diverges,
guarantee the existence of a natural $n_0$ such that
\begin{equation}\label{ecu: bbii}
\l<\dist_{(\overline{M(\J_n)},X_n)}(p,\J_n),\quad \forall p\in
\J^\mu,\quad \forall n\geq n_0.
\end{equation}
Define the multicycle $\gorro{\J}:=\J_{n_0}$ and the conformal
minimal immersion
$\gorro{X}:=X_{n_0}:\overline{M(\gorro{\J})}\to\R^3.$ Then
(2$_{n_0}$) implies ({L\ref{lem: gafa-4}}.1) and ({L\ref{lem:
gafa-4}}.2). Properties ({L\ref{lem: gafa-4}}.3) and ({L\ref{lem:
gafa-4}}.4) trivially hold from (1$_{n_0}$) and (4$_{n_0}$),
respectively. Inequality \eqref{ecu: bbii} implies ({L\ref{lem:
gafa-4}}.5). Finally, taking into account \eqref{ecu: gafa-5.7},
(6$_{n_0}$) and \eqref{ecu: gafa-3.6}, we obtain
\[
\|X(p)-\gorro{X}(p)\|\leq
\|X(p)-X(\pp_i(p))\|+\|X(\pp_i(p))-\gorro{X}(p)\|<
\]
\[\frac{\mu}2+r_{n_0}<\mu,\quad
\forall p\in \gorro{\g}_i,\quad \forall i=1,\ldots,\ee.
\]
Hence $\|X(p)-\gorro{X}(p)\|<\mu$ for any $p\in \gorro{\J}.$
Therefore the Maximum Principle for harmonic maps guarantees that
this inequality holds for any $p\in \overline{M(\gorro{\J})}.$
This proves Property ({L\ref{lem: gafa-4}}.6). Finally,
({L\ref{lem: gafa-4}}.7) follows from ({L\ref{lem: gafa-4}}.6),
\eqref{ecu: gafa-5.7} and (6$_{n_0}$).
\end{proof}

%%%%%%%%%%%%%%%%%%%%%%%%%%%%%%%%%%%%%%%%%%%%%%%%%%%%%%%%%%%%%%%%%%%%%%%%%%%%%%%%%%%%%%%%%%%%%%%%%%%%%%%%%%%%%%%%%%%%%%%%%%%
%%%%%%%%%%%%%%%%%%%%%%%%%%%%%%%%%%%%%%%%%%%%%%%%%%%%%%%%%%%%%%%%%%%%%%%%%%%%%%%%%%%%%%%%%%%%%%%%%%%%%%%%%%%%%%%%%%%%%%%%%%%
%%%%%%%%%%%%%%%%%%%%%%%%%%%%%%%%%%%%%%%%%%%%%%%%%%%%%%%%%%%%%%%%%%%%%%%%%%%%%%%%%%%%%%%%%%%%%%%%%%%%%%%%%%%%%%%%%%%%%%%%%%%

\section{Compact complete minimal immersions in $\R^3$}

Now, we are able to prove the theorems stated in the introduction.
They follow from the following one.

\begin{theorem}\label{teo: existe}
Let $\J$ be a multicycle on the Riemann surface $M',$ and
$\phi:\overline{M(\J)}\to\R^3$ a conformal minimal immersion.
Then, for any $\ep>0,$ there exist a domain $M_\ep$ and a
conformal complete minimal immersion $\phi_\ep:M_\ep\to\R^3$
satisfying
\begin{enumerate}[\rm (I)]
\item $\overline{M(\J^\ep)}\subset M_\ep\subset \overline{M_\ep}\subset M(\J),$
and $M_\ep$ has the same topological type as $M(\J).$

\item $\phi_\ep$ admits a continuous extension
$\Phi_\ep:\overline{M_\ep}\to \R^3.$

\item $\|\phi-\Phi_\ep\|<\ep$ in $\overline{M_\ep}.$

\item ${(\Phi_\ep)}_{|\Gamma}$ is an embedding, for any
$\Gamma$ connected component of $\partial M_\ep.$

\item $\de^H\big(\phi(\overline{M(\J)}), \Phi_\ep(\overline{M_\ep})
\big)<\ep.$

\item The Hausdorff dimension of $\Phi_\ep(\Gamma)$ is $1,$ for
any $\Gamma$ connected component of $\partial M_\ep.$
\end{enumerate}
\end{theorem}

\subsection{Proof of Theorem \ref{teo: existe}}

Assume $\J$ consists of $\ee$ cycles. Label $\ve:=\ep/2$ and
consider a multicycle $\J'=\{\g_1,\ldots,\g_\ee\}$ satisfying
\begin{enumerate}[\rm (a)]
\item $\J^\ve<\J'<\J.$
\item $\phi(\g_i)$ is a Jordan curve, for any $i=1,\ldots,\ee.$
\item $\de^H\big( \phi(\overline{M(\J)}), \phi(\overline{M(\J')}) \big)<\ve.$
\end{enumerate}

We use Lemma \ref{lem: gafa-4} to construct, $\forall n\in\N,$ a
family $\chi_n=\{\J_n,\mathcal{T}(\J_n),X_n,\ep_n,\xi_n,\tau_n\},$
where
\begin{enumerate}[1]
\item[$\bullet$] $\J_n=\{\g_{n,1},\ldots,\g_{n,\ee}\}$ is a multicycle on $M'.$
\item[$\bullet$]
$\mathcal{T}(\J_n)=\{\mathcal{T}(\g_{n,1}),\ldots,\mathcal{T}(\g_{n,\ee})\},$
where $\mathcal{T}(\g_{n,i})$ is a tubular neighborhood of
$\g_{n,i}$ in $\overline{M(\J_n)},$ for $i=1,\ldots,\ee.$
\item[$\bullet$] $X_n:\overline{M(\J_n)}\to\R^3$ is a conformal minimal immersion.
\item[$\bullet$] $\{\ep_n\}_{n\in\N}$ and $\{\xi_n\}_{n\in\N}$ are decreasing sequences
of positive real numbers with $\xi_n<\ep_n<6\ve/\pi^2 n^2.$
\item[$\bullet$] $\{\tau_n\}_{n\in\N}$ is an increasing sequence of
positives which diverges to $+\infty.$
\end{enumerate}
Moreover, we will construct the sequence $\{\chi_n\}_{n\in\N}$ so that the following list of properties holds:
\begin{enumerate}[\rm (A$_{n}$)]
\item $X_n(\g_{n,i})$ is a Jordan curve, for any
$i=1,\ldots,\ee.$

\item $\mathcal{T}(\g_{n,i})\subset \mathcal{T}(\g_{n-1,i}),$ $\forall i=1,\ldots,\ee.$

\item $\J_n^{\xi_n}\subset \mathcal{T}(\J_n),$ i.e.,
$\g_{n,i}^{\xi_n}\subset \mathcal{T}(\g_{n,i}),$ for any
$i=1,\ldots,\ee.$

\item $\J_{n-1}^{\xi_{n-1}}< \J_{n-1}^{\ep_n}< \J_n^{\xi_n}<\J_n<\J_{n-1}.$

\item $\|X_n-X_{n-1}\|<\ep_n$ in $\overline{M(\J_n)}.$

\item
$(X_n)_{|\overline{\mathcal{T}(\g_{n,i})}}:\overline{\mathcal{T}(\g_{n,i})}\to\R^3$
is an embedding, for any $i=1,\ldots,\ee.$

\item $\|X_n-X_{n-1}\|<\min\{\eta_{n-1},1/(n^2\tau_{n-1}^n)\},$
in $\overline{\mathcal{T}(\J_n)}\subset\overline{M(\J_n)},$ where
\[
\eta_{n-1}:=\min\left.\left\{
\Psi\left(\overline{\mathcal{T}(\g_{n-1,i})}\;,\;
(X_{n-1})_{|\overline{\mathcal{T}(\g_{n-1,i})}}\;,\;n\right)\;\right|\;i=1,\dots,\ee\right\}.
\]
Recall the definition of $\Psi$ in page \pageref{def: infimo}.

\item $1/\ep_n< \dist_{(\overline{M(\J_n^{\xi_n})},X_n)}(\J_{n-1}^{\xi_{n-1}}, \J_n^{\xi_n}).$

\item $ds_{X_n}\geq \alpha_n\cdot ds_{X_{n-1}}$ in $\overline{M(\J_{n-1}^{\xi_{n-1}})},$
where the sequence $\{\a_k\}_{k\in\N}$ is given by
\[
\a_1:=\frac12\, e^{1/2},\quad \a_k:=e^{-1/2^k}\text{ for }k>1.
\]
Notice that $0<\a_k<1$ and $\{\prod_{m=1}^k \a_m\}_{k\in\N}$
converges to $1/2.$

\item $\de^H\big( X_n(\overline{M(\J_n)}), X_{n-1}(\overline{M(\J_{n-1})})  \big)<\ep_n.$

\item $\forall i=1,\ldots,\ee,$ there exist points
$\a_{1,i}^n,\ldots,\a_{N_n,i}^n\in X_n(\g_{n,i})$ such that
\[
X_n(\overline{\mathcal{T}(\g_{n,i})})\subset
\bigcup_{j=1}^{N_n}\B(\a_{j,i}^n,s_n),
\]
where $s_n=1/\tau_n^n,$ $N_n$ is the integer part of $\tau_n/s_n,$
and $\B(\a_{j,i}^n,s_n)=\{y\in\R^3\;|\;\|\a_{j,i}^n-y\|<s_n\}.$

\end{enumerate}

For the first family we choose $\J_1=\J'$ and $X_1=\phi.$ Hence
Property (b) of $\J'$ guarantees (A$_1$). Notice that we can
choose $\mathcal{T}(\J_1)$ close enough to $\J_1$ so that (F$_1$)
holds. This is possible because of (A$_1$) and the local
injectivity of $X_1.$ Define
\[
\tau_1:=\max\{1\;,\;\max\{\longui(X_1(\g_{1,i}))\;|\;i=1,\ldots,\ee\}\}.
\]
For $i=1,\ldots,\ee,$ let $\a_{1,i}^1,\ldots,\a_{N_1,i}^1$ be
points on $X_1(\g_{1,i})$ with mutual distance along
$X_1(\g_{1,i}).$ The choice of $\tau_1$ implies that this distance
is equal or smaller than $s_1.$ Hence, the union of
$\B(\a_{j,i}^1,s_1)$ is an open set in $\R^3$ which covers the
curve $X_1(\g_{1,i}).$ Therefore we can choose $\mathcal{T}(\J_1)$
close enough to $\J_1$ so that (K$_1$) holds. Finally, consider
positive numbers $\xi_1<\ep_1<6\ve/\pi^2,$ where $\xi_1$ is small
enough so that (C$_1$) holds. Properties (B$_1$), (D$_1$),
(E$_1$), (G$_1$), (H$_1$), (I$_1$) and (J$_1$) do not make sense.

Assume that we have constructed $\chi_1,\ldots,\chi_n.$ In order
to define $\chi_{n+1}$ we consider a sequence
$\{\til{\ep}_m\}_{m\in\N}$ decreasing to zero with
\begin{equation}\label{ecu: peque}
\til{\ep}_m<\min\left\{\frac{6\ve}{\pi^2(n+1)^2}\;,\;\eta_{n}\;,\;
\frac{1}{ \tau_n^{n+1}(n+1)^2} \right\},\quad \forall m\in\N.
\end{equation}
Notice that (F$_{n}$) guarantees that $\eta_{n}>0,$ so the choice
of $\til{\ep}_m$ is possible. Now, apply Lemma \ref{lem: gafa-4}
to the data
\[
\J=\J_n,\quad X=X_n,\quad \l=\frac1{\til{\ep}_m},\quad
\mu=\til{\ep}_m.
\]
Hence, for any $m\in\N,$ we obtain a multicycle $\Gamma_m,$ and a
conformal minimal immersion $Y_m:\overline{M(\Gamma_m)}\to\R^3.$
Statements ({L\ref{lem: gafa-4}}.4) and ({L\ref{lem: gafa-4}}.6)
in Lemma \ref{lem: gafa-4} imply that, if $m$ is large enough,
$\J_n^{\xi_n}< \Gamma_m$ and the sequence $\{Y_m\}_{m\in\N}$
uniformly converges in $\overline{M(\J_n^{\xi_n})}$ to $X_n.$ In
particular, the sequence $\{ds_{Y_m}\}_{m\in\N}$ uniformly
converges in $\overline{M(\J_n^{\xi_n})}$ to $ds_{X_n}.$ Therefore
there exists $m_0\in\N$ large enough so that
\begin{equation}\label{ecu: gafa-7.2}
\J_n^{\xi_n}<\J_n^{\til{\ep}_{m_0}}<\Gamma_{m_0},
\end{equation}
\begin{equation}\label{ecu: gafa-7.3}
ds_{Y_{m_0}}\geq \alpha_{n+1}\cdot ds_{X_n}\quad\text{in
}\overline{M(\J_n^{\xi_n})}.
\end{equation}
Define $\J_{n+1}:=\Gamma_{m_0},$ $X_{n+1}:=Y_{m_0},$ and
$\ep_{n+1}:=\til{\ep}_{m_0}.$ Hence Statement ({L\ref{lem:
gafa-4}}.2) guarantees (A$_{n+1}$). Then we choose
$\mathcal{T}(\J_{n+1}),$ a tubular neighborhood of $\J_{n+1}$ in
$\overline{M(\J_{n+1})},$ close enough to $\J_{n+1}$ so that
(B$_{n+1}$) and (F$_{n+1}$) hold. This choice is possible because
of ({L\ref{lem: gafa-4}}.3), (A$_{n+1}$) and the local injectivity
of $X_{n+1}.$  Define
\[
\tau_{n+1}:=\max\{\tau_n+1\;,\;\max\{\longui(X_{n+1}(\g_{n+1,i}))\;|\;i=1,\ldots,\ee\}\}.
\]
For $i=1,\ldots,\ee,$ let
$\a_{1,i}^{n+1},\ldots,\a_{N_{n+1},i}^{n+1}$ be points on
$X_{n+1}(\g_{n+1,i})$ with mutual distance along
$X_{n+1}(\g_{n+1,i}).$ The choice of $\tau_{n+1}$ implies that
this distance is not larger than $s_{n+1}.$ Therefore the union of
$\B(\a_{j,i}^{n+1},s_{n+1})$ covers the curve
$X_{n+1}(\g_{{n+1},i}).$ Hence we can choose
$\mathcal{T}(\J_{n+1})$ close enough to $\J_{n+1}$ so that
(K$_{n+1}$) holds.

Now, taking into
account \eqref{ecu: gafa-7.2} and ({L\ref{lem: gafa-4}}.5),
we conclude
\[
\frac1{\ep_{n+1}}<
\dist_{(\overline{M(\J_{n+1})},X_{n+1})}(\J_n^{\xi_n},\J_{n+1}).
\]
Then we can choose $\xi_{n+1}\in ]0,\ep_{n+1}[$ small enough so
that (C$_{n+1}$), (D$_{n+1}$) and (H$_{n+1}$) hold. Moreover,
\eqref{ecu: gafa-7.3} implies (I$_{n+1}$). ({L\ref{lem:
gafa-4}}.6) and \eqref{ecu: peque} imply (E$_{n+1}$) and
(G$_{n+1}$). Finally ({L\ref{lem: gafa-4}}.7) guarantees
(J$_{n+1}$). In this way we have finished the construction of the
sequence $\{\chi_n\}_{n\in\N}.$

Now, define the set
\[
M_\ep:=\bigcup_{n\in\N} M(\J_n^{\ep_{n+1}}) =\bigcup_{n\in\N}
M(\J_n^{\xi_n}).
\]
From (D$_n$), $n\in \N,$ we conclude that the set $M_\ep$ is an
expansive union of domains with the same topological type as
$M(\J).$ Therefore elementary topological arguments give us that
$M_\ep$ is a domain with the same topological type as $M(\J).$ On
the other hand, (D$_n$), $n\in \N,$ also imply that
\[
\overline{M_\ep}=\bigcap_{n\in\N} \overline{M(\J_n)}.
\]
Then Properties (E$_n$), $n\in\N,$ guarantee that
$\{X_n\}_{n\in\N}$ is a Cauchy sequence uniformly in
$\overline{M_\ep}$ of continuous maps. Therefore it converges to a
continuous map $\Phi_\ep:\overline{M_\ep}\to\R^3.$ Define
$\phi_\ep:=(\Phi_\ep)_{|M_\ep}:M_\ep\to\R^3.$

Let us check that $\phi_\ep$ satisfies the conclusion of the
theorem.

\noindent$\bullet$ Since Properties (I$_n$), $n\in\N,$ $\phi_\ep$ is a
conformal minimal immersion.

\noindent$\bullet$ Completeness of $\phi_\ep$ follows from Properties
(H$_n$), (I$_n$) and the fact that the sequence
$\{1/\ep_n\}_{n\in\N}$ diverges.

\noindent$\bullet$ Statements (I) and (II) trivially hold (in order to
check (I), take into account Property (a) of $\J'$).

\noindent$\bullet$ Notice that $\sum_{n=1}^\infty
\ep_n<\ve=\ep/2.$ Taking this fact into account, Statement (III)
follows from (E$_n$), $n\in\N,$ and Statement (V) is implied by
(J$_n$), $n\in\N,$ and Property (c) of the multicycle $\J'.$

\noindent$\bullet$ In order to prove (IV) we have to work a little
further. From Statement (I) we obtain that $\partial M_\ep$ has
exactly $\ee$ connected components; one of them in each connected
component of $M(\J)\setminus M(\J^\ep).$ Consider two different
points, $p$ and $q,$ in the same connected component $\Gamma$ of
$\partial M_\ep.$ Assume that $\Gamma$ is the connected component
of $\partial M_\ep$ related with the $i$-th cycle of $\J.$ Then
Properties (B$_n$), (D$_n$), $n\in\N$, and the definition of
$M_\ep,$ imply that
\begin{equation}\label{ecu: dentro}
\{p,q\}\subset\G\subset \overline{\mathcal{T}(\g_{n,i})},\quad
\forall n\in\N.
\end{equation}
On the other hand, there exists $n_0\in\N$ such that
\begin{equation}\label{ecu: >1/n}
\dist_{(M',ds)}(p,q)>1/n, \quad\forall n\geq n_0.
\end{equation}
Then, for any $n> n_0,$ from Property (G$_n$) and the definition
of the constant $\eta_{n-1},$ we have
\[
\|X_{n-1}(p)-X_{n-1}(q)\|\leq
\]
\[
\|X_{n-1}(p)-X_n(p)\|+\|X_{n-1}(q)-X_n(q)\|+\|X_n(p)-X_n(q)\|<
\]
\[
2\eta_{n-1}+\|X_n(p)-X_n(q)\|\leq
\]
\[
\frac1{n^2}\cdot\|X_{n-1}(p)-X_{n-1}(q)\| + \|X_n(p)-X_n(q)\|,
\]
where we have taken \eqref{ecu: dentro} and \eqref{ecu: >1/n} into
account to bound $\eta_{n-1}.$ Hence
\[
\|X_n(p)-X_n(q)\|>\left( 1-\frac1{n^2}\right)\cdot
\|X_{n-1}(p)-X_{n-1}(q)\|,\quad \forall n> n_0.
\]
Then, for any $k>0,$ the above inequalities guarantee
\begin{equation}\label{ecu: cota-f}
\|X_{n_0+k}(p)-X_{n_0+k}(q)\|> \|X_{n_0}(p)-X_{n_0}(q)\|\cdot
\prod_{m=n_0+1}^{n_0+k}\left( 1-\frac1{m^2} \right).
\end{equation}
If we take limits in \eqref{ecu: cota-f} as
$k\to\infty,$ then we obtain that
\[
\|\Phi_\ep(p)-\Phi_\ep(q)\|\geq \frac12
\|X_{n_0}(p)-X_{n_0}(q)\|>0,
\]
where we have taken into account (F$_{n_0}$). This inequality
proves Statement (IV).

\noindent$\bullet$ Finally, let us check Statement (VI). Fix
$n\in\N$ so that $\sum_{k=n}^\infty 1/k^2<1.$ Let $\Gamma$ be the
connected component of $\partial M_\ep$ related with the $i$-th
cycle of $\J,$ and consider $p\in \Gamma.$ Notice that
$p\in\overline{\mathcal{T}(\g_{m,i})},$ $\forall m\in\N.$ Then,
for any $k>n,$ since (G$_m$), $m=n+1,\ldots,k,$ we deduce
\[
\|X_k(p)-X_{n}(p)\|\leq \frac1{k^2 \tau_{k-1}^k}+ \frac1{(k-1)^2
\tau_{k-2}^{k-1}}+\cdots +\frac1{(n+1)^2 \tau_n^{n+1}}.
\]
Then, using that $\{\tau_m\}_{m\in\N}$ is increasing and
$\tau_n\geq 1,$ we obtain
\begin{equation}\label{ecu: dimen}
\|X_k(p)-X_n(p)\|< \left( \sum_{m=n+1}^\infty \frac1{m^2}
\right)\frac1{\tau_n^n} < \frac1{\tau_n^n}.
\end{equation}
If we take limits in \eqref{ecu: dimen} as $k$ goes to infinity,
then we have
\begin{equation}\label{ecu: dimen2}
\|\Phi_\ep(p)-X_n(p)\|<\frac1{\tau_n^n}=s_n,\quad \text{for any
$n$ large enough}.
\end{equation}
On the other hand, Property (K$_n$)
guarantees
\[
X_n(p)\in \bigcup_{j=1}^{N_n}\B(\a_{j,i}^n,s_n),\quad \forall p\in
\Gamma.
\]
Hence, taking into account \eqref{ecu: dimen2}, we have
\[
\Phi_\ep(\Gamma)\subset \bigcup_{j=1}^{N_n}\B(\a_{j,i}^n,2s_n).
\]
Moreover, it is clear that $N_n\cdot (2s_n)^{1+1/n}<4.$ This
inequality particularly means that the Hausdorff measure
$\mathcal{H}^1(\Phi_\ep(\Gamma))<\infty$ and so the Hausdorff
dimension of $\Phi_\ep(\Gamma)$ is at most $1$ (see \cite{M} for
details about Hausdorff measure). However, the Hausdorff dimension
of $\Phi_\ep(\Gamma)$ is at least $1$ (it is a curve).

This proves Statement (VI) and concludes the proof of the theorem.

\subsection{Some consequences of Theorem \ref{teo: existe}}

In this subsection we remark some results that follow
straightforwardly from Theorem \ref{teo: existe}.

The first one is an improvement of Theorem \ref{teo: existe}. We
can trivially refine the behavior of the limit set to obtain the
embeddedness  of $\Phi_\ep$ along $\partial M_\ep.$

\begin{corollary}\label{cor: teo}
Let $\J$ be a multicycle on the Riemann surface $M',$ and
$\phi:\overline{M(\J)}\to\R^3$ a conformal minimal immersion.
Then, for any $\ep>0,$ there exist a domain $M_\ep$ and a
conformal complete minimal immersion $\phi_\ep:M_\ep\to\R^3$
satisfying
\begin{enumerate}[1]
\item[$\bullet$] $\overline{M(\J^\ep)}\subset M_\ep\subset \overline{M_\ep}\subset M(\J),$
and $M_\ep$ has the same topological type as $M(\J).$

\item[$\bullet$] $\phi_\ep$ admits a continuous extension
$\Phi_\ep:\overline{M_\ep}\to \R^3.$

\item[$\bullet$] $\|\phi-\Phi_\ep\|<\ep$ in $\overline{M_\ep}.$

\item[$\bullet$] $(\Phi_\ep)_{|\partial M_\ep}$ is an embedding.

\item[$\bullet$] $\de^H\big(\phi(\overline{M(\J)}), \Phi_\ep(\overline{M_\ep})
\big)<\ep.$

\item[$\bullet$] The Hausdorff dimension of $\Phi_\ep(\partial M_\ep)$ is $1.$
\end{enumerate}
\end{corollary}

The next result is a density type theorem. We prove that the set
of finite families of curves in $\R^3$ spanned by complete
(connected) minimal surfaces is dense in the set of finite
families of curves spanned by (connected) minimal surfaces, with
the Hausdorff metric.

\begin{corollary}\label{cor: Jordan}
Let $\Sigma$ be a finite family of closed curves in $\R^3$ so that
the Plateau problem for $\Sigma$ admits a solution. Then, for any
$\xi>0,$ there exist a compact Riemann surface $\mathcal{M},$ an open domain $M\subset \mathcal{M}$ and a continuous
map $\Phi:\overline{M}\to\R^3$ such that
\begin{enumerate}[1]
\item[$\bullet$] $\Phi_{|M}$ is a conformal complete minimal immersion.

\item[$\bullet$] $\Phi_{|\partial M}$ is an embedding.

\item[$\bullet$] $\de^H(\Sigma,\Phi(\partial M))<\xi.$

\item[$\bullet$] The Hausdorff dimension of $\Phi(\partial M)$ is $1.$
\end{enumerate}
\end{corollary}

The next result shows that compact complete minimal immersions are
not rare. Before state it, recall that any Riemann surface with
finite topology and analytic boundary can be seen as the closure
of an open region of a compact Riemann surface \cite{AS}.

\begin{corollary}\label{cor: densidad}
Complete minimal surfaces are dense in the space of minimal
surfaces spanning a finite set of closed curves in $\R^3,$ endowed
with the topology of the Hausdorff distance.
\end{corollary}

Theorem \ref{teo: existe} can be seen as an improvement of the following result
\cite[Theorem 2]{AFM}.

\begin{corollary}\label{cor: gafa-teo-2}
Let $\J$ be a multicycle on $M'$ and
$\phi:\overline{M(\J)}\to\R^3$ a conformal minimal immersion.
Then, for any $\ep>0,$ there exists a domain $M_\ep$ in $M',$ with
$\overline{M(\J)}\subset M_\ep,$ and there exists a conformal
complete minimal immersion $\phi_\ep:M_\ep\to\R^3$ such that
\begin{enumerate}[1]
\item[$\bullet$] $\|\phi_\ep-\phi\|<\ep$ in $\overline{M(\J)}.$
\item[$\bullet$] $\de^H\big( \phi(\overline{M(\J)}), \overline{\phi_\ep(M_\ep)} \big)<\ep.$
\end{enumerate}
\end{corollary}

%%%%%%%%%%%%%%%%%%%%%%%%%%%%%%%%%%%%%%%%%%%%%%%%%%%%%%%%%%%%%%%%%%%%%%%%%%%%%%%%%%%%%%%%%%%%%%%%%%%%%%%%%%%%%%%%%%%%%%%%%%%%%%%%%%%%%%%%%%%%%%%%%%%%%%%
%%%%%%%%%%%%%%%%%%%%%%%%%%%%%%%%%%%%%%%%%%%%%%%%%%%%%%%%%%%%%%%%%%%%%%%%%%%%%%%%%%%%%%%%%%%%%%%%%%%%%%%%%%%%%%%%%%%%%%%%%%%%%%%%%%%%%%%%%%%%%%%%%%%%%%%%
%%%%%%%%%%%%%%%%%%%%%%%%%%%%%%%%%%%%%%%%%%%%%%%%%%%%%%%%%%%%%%%%%%%%%%%%%%%%%%%%%%%%%%%%%%%%%%%%%%%%%%%%%%%%%%%%%%%%%%%%%%%%%%%%%%%%%%%%%%%%%%%%%%%%%%%%

\bibliographystyle{amsplain}

\begin{thebibliography}{10}

\bibitem{AS} L. V. Ahlfors and L. Sario, {\em Riemann Surfaces.}
Princeton University Press, Princeton, New Jersey (1974).

\bibitem{AFM} A. Alarc\'{o}n, L. Ferrer and F. Mart\'{i}n, {\em Density
theorems for complete minimal surfaces in $\R^3.$} Geom. Funct.
Anal. {\bf 18} (1) (2008), 1--49.

\bibitem{AN} A. Alarc\'{o}n and N. Nadirashvili, {\em Limit sets
for complete minimal immersions.} Math. Z. {\bf 258} (1) (2008),
107--113.

\bibitem{C} E. Calabi, {\em Problems in differential geometry}.
Ed. S. Kobayashi and J. Ells, Jr., Proceedings of the United
States-Japan Seminar in Differential Geometry, Kyoto, Japan, 1965.
Nippon Hyoronsha Co., Ltd., Tokyo 170 (1966).


\bibitem{CM} T. H. Colding and W. P. Minicozzi, {\em The Calabi-Yau conjectures for embedded
surfaces.} Ann. of Math. (2) {\bf 167} (1) (2008), 211--243.

\bibitem{D} J. Douglas, {\em Solution of the problem of
Plateau.} Trans. Amer. Math. Soc. {\bf 33} (1) (1931), 263--321.

\bibitem{FMM} L. Ferrer, F. Mart\'{i}n and W. H. Meeks III, {\em Existence of proper minimal surfaces of arbitrary topological type.} Preprint.

\bibitem{LMM} F. J. L\'{o}pez,  F. Mart\'{i}n and S. Morales,
{\em Adding handles to Nadirashvili's surfaces.} J. Diff. Geom.
{\bf 60} (1) (2002), 155--175.

\bibitem{MMN} F. Mart\'{i}n, W. H. Meeks III and N. Nadirashvili,
{\em Bounded domains which are universal for minimal surfaces.}
Amer. J. Math. {\bf 129} (2) (2007), 455--461.

\bibitem{MM1} F. Mart\'{i}n and S. Morales, {\em On the asymptotic
behavior of a complete bounded minimal surface in $\R^3$.} Trans.
Amer. Math. Soc. {\bf 356} (10) (2004), 3985--3994.

\bibitem{MM2} F. Mart\'{i}n and S. Morales, {\em Complete proper minimal surfaces in convex bodies of $\R^3$ (II):
The behavior of the limit set.} Comment. Math. Helv. {\bf 81} (3)
(2006), 699--725.

\bibitem{MN} F. Mart\'{i}n and N. Nadirashvili, {\em A Jordan curve
spanned by a complete minimal surface.} Arch. Ration. Mech. Anal.
{\bf 184} (2) (2007), 285--301.

\bibitem{M} F. Morgan, {\em Geometric Measure Theory. A
Beginner's guide. Third edition.} Academic Press, Inc., San Diego,
CA (200).

\bibitem{N1} N. Nadirashvili, {\em Hadamard's and Calabi-Yau's
conjectures on negatively curved and minimal surfaces.} Invent.
Math. {\bf 126} (3) (1996), 457--465.

\bibitem{N2} N. Nadirashvili, {\em An application of potential analysis to
minimal surfaces.} Mosc. Math. J. {\bf 1} (4) (2001), 601--604.

\bibitem{R} T. Rad\'{o}, {\em On Plateau's problem.} Ann. of Math. (2) {\bf 31} (3) (1930), 457--469.

\end{thebibliography}

\end{document}